\documentclass{amsart}

\usepackage{amsmath}
\usepackage{amssymb}
\usepackage{MnSymbol}
\usepackage{dsfont}
\usepackage{epsfig}  		% For postscript
\usepackage{xypic}
\usepackage{bbm}
\usepackage{hyperref}
\usepackage{eucal}
\usepackage{mathrsfs}

\usepackage{xcolor}

\newtheorem*{gelfondschneider}{Gelfond-Schneider Theorem}

\newtheorem*{schanuel}{Schanuel's Conjecture}

\newtheorem{thm}{Theorem}[section]
\newtheorem{prop}[thm]{Proposition}
\newtheorem{lem}[thm]{Lemma}
\newtheorem{cor}[thm]{Corollary}
\newtheorem{conjecture}[thm]{Conjecture}
\newtheorem{observation}[thm]{Observation}

\theoremstyle{definition}
\newtheorem{definition}[thm]{Definition}
\newtheorem{example}[thm]{Example}

\theoremstyle{remark}

\newtheorem{remark}{Remark}[section]

\numberwithin{equation}{section}

\newcommand{\CC}{\mathds{C}}
\newcommand{\RR}{\mathds{R}}
\newcommand{\QQ}{\mathds{Q}}

\newcommand{\ZZ}{\mathds{Z}} 
 
\newcommand{\NN}{\mathds{N}}

\newcommand{\ad}{\mathrm{ad}}
\newcommand{\Ad}{\mathrm{Ad}}

\newcommand{\Iso}{\group{Iso}}

\newcommand{\SL}{\group{SL}}

\newcommand{\SO}{\group{SO}}

\newcommand{\Aut}{\group{Aut}}

\newcommand{\group}{\mathrm} 

\newcommand{\ac}[1]{\overline{#1}^{\rm z}}

\renewcommand{\hat}{\widehat}
\renewcommand{\rho}{\varrho}

\renewcommand{\bar}{\overline}

\renewcommand{\epsilon}{\varepsilon}

\newcommand{\Z}{\mathrm{Z}}

\newcommand{\zsp}{\mathbf{0}} 
\newcommand{\met}{\langle\cdot,\cdot\rangle}
\newcommand{\one}{\{e\}}
\renewcommand{\Re}{\mathrm{Re}}
\renewcommand{\Im}{\mathrm{Im}}

\newcommand{\tensor}{\mathsf}
\newcommand{\Ric}{\tensor{Ric}}

\newcommand{\g}{\tensor{g}}

\newcommand{\tr}{\mathrm{tr}}

\DeclareMathOperator{\im}{\mathrm{im}}

\DeclareMathOperator{\Span}{\mathrm{span}}

\newcommand{\frg}{\mathfrak{g}}
\newcommand{\frh}{\mathfrak{h}}
\newcommand{\fra}{\mathfrak{a}}
\newcommand{\frb}{\mathfrak{b}}

\newcommand{\frn}{\mathfrak{n}}

\newcommand{\frw}{\mathfrak{w}}
\newcommand{\frk}{\mathfrak{k}}
\newcommand{\frj}{\mathfrak{j}}
\newcommand{\fri}{\mathfrak{i}}

\newcommand{\fru}{\mathfrak{u}}
\newcommand{\frs}{\mathfrak{s}}

\newcommand{\frr}{\mathfrak{r}}
\newcommand{\frv}{\mathfrak{v}}

\newcommand{\zen}{\mathfrak{z}}

\newcommand{\sso}{\mathfrak{so}}

\newcommand{\der}{\mathfrak{der}}
\newcommand{\ab}{\mathfrak{ab}}

\newcommand{\I}{\mathrm{I}}

\renewcommand{\i}{\mathrm{i}}
\newcommand{\e}{\mathrm{e}}
\renewcommand{\phi}{\varphi}

\begin{document}

\title[Compact pseudo-Riemannian homogeneous Einstein manifolds]{Compact pseudo-Riemannian homogeneous Einstein manifolds of low dimension}

\author[Globke]{Wolfgang Globke}
\address{Wolfgang Globke, School of Mathematical Sciences, The University of Adelaide, SA 5005, Australia}
\email{wolfgang.globke@adelaide.edu.au}

\author[Nikolaevsky]{Yuri Nikolayevsky}
\address{Yuri Nikolayevsky,
Department of Mathematics and Statistics,
La Trobe University,
Melbourne,
VIC 3086,
Australia}
\email{y.nikolayevsky@latrobe.edu.au}

\thanks{Wolfgang Globke was supported by the Australian Research Council grant {DE150101647}.
Yuri Nikolayevsky was partially supported by the
Australian Research Council
Discovery grant DP130103485.
}

\subjclass[2010]{Primary 53C30; Secondary 17B30, 53C50, 22E25, 57S20}

\begin{abstract}
Let $M$ be pseudo-Riemannian homogeneous Einstein manifold of
finite volume, and suppose a connected Lie group $G$ acts
transitively and isometrically on $M$.
In this situation, the metric on $M$
induces a bilinear form $\met$ on the Lie algebra $\frg$ of $G$
which is nil-invariant, a property closely related to invariance.
We study such spaces $M$ in three important cases.
First, we assume $\met$ is invariant, in which
case the Einstein property requires that $G$ is either solvable
or semisimple.
Next, we investigate the case where $G$ is solvable. Here,
$M$ is compact and $M=G/\Gamma$ for a lattice $\Gamma$ in $G$.
We show that in dimensions less or equal to $7$,
compact quotients $M=G/\Gamma$
exist only for nilpotent groups $G$.
We conjecture that this is true for any dimension.
In fact, this holds if Schanuel's Conjecture on transcendental
numbers is true.
Finally, we consider semisimple Lie groups $G$,
and find that $M$ is covered by a pseudo-Riemannian product of Einstein
manifolds corresponding to the compact and the non-compact factors of $G$.
\end{abstract}

\maketitle

%
%{\tt Version MAIN of \textcolor{blue}{\bf \today}.}
%

\baselineskip=13pt

\tableofcontents

% !TEX root = einsteinCompactSolvable.tex

%%%%%%%%%%%%%%%%%%%%%%%
\section{Introduction and main results}
%%%%%%%%%%%%%%%%%%%%%%%
\label{sec:intro}

Let $M$ be a pseudo-Riemannian homogeneous manifold with finite
volume, such that $M=G/H$, where $G$ is a connected Lie group
with a closed subgroup $H$, and $G$
acts isometrically and almost effectively on $M$
(meaning $H$ contains no connected normal subgroup of $G$).
We further assume that the pseudo-Riemannian metric $\g_M$ on $M$
is an Einstein metric, that is, the Ricci tensor is a multiple
of the metric, $\Ric=\lambda\g_M$, for some real constant $\lambda$.
This work is motivated by questions on the existence of such
spaces $M$, and the algebraic structure of the groups $G$.

The pseudo-Riemannian metric $\g_M$ on $M$ pulls back to a
symmetric
bilinear tensor on $G$, and thus induces a symmetric bilinear
form $\met$ on the Lie algebra $\frg$ of $G$.
We review some basic notions of scalar products on Lie algebras in
Section \ref{sec:scalar}.

The fact that $M$ has finite volume implies that $\met$ is
nil-invariant (Definition \ref{def:nilinvariant}), a property
closely related to invariance of $\met$
(see Baues and Globke \cite{BG}).
It was shown by Baues and Globke \cite[Theorem 1.2]{BG} that
a nil-invariant symmetric bilinear form $\met$ on a
finite-dimensional solvable Lie algebra must be invariant.
The structure of solvable Lie algebras with invariant scalar
pro\-duct can be understood through a sequential reduction to
lower-dimensional algebras of the same type.
This process is briefly reviewed in Section \ref{sec:reduction}.

In Section \ref{sec:nilinvariantEinstein} we consider Lie groups
with Einstein metrics. If the Einstein metric is bi-invariant,
we can make an easy but very important observation:

\begin{prop}\label{prop:Einstein_solvable}
If $\g$ is a bi-invariant Einstein metric on $G$,
then
\begin{enumerate}
\item[(a)]
either $G$ is semisimple and $\g$ is a non-zero multiple of the Killing
form,
\item[(b)]
or $G$ is solvable and its Killing form vanishes.
\end{enumerate}
\end{prop}

Combining this proposition with results from \cite{BGZ}
shows that for any pseudo-Riemannian homogeneous Einstein manifold
$M=G/H$ of finite volume for a Lie group $G$ whose Levi subgroup
has no compact factors, $G$ is either semisimple or solvable
(Corollary \ref{cor:alternative1}),
and the stabilizer $H$ is a lattice in $G$.

We study pseudo-Riemannian homogeneous Einstein manifolds $M=G/H$
of finite volume and with $G$ a connected solvable Lie group
in Section \ref{sec:solvable}.
The stabilizer $H$ is a lattice in $G$, and
this implies that the induced bilinear form $\met$ on $\frg$ is
in fact a scalar product.
Moreover, $\met$ is invariant and therefore the Killing form of
$\frg$ vanishes in order to satisfy the Einstein condition.
This is trivially satisfied by any nilpotent Lie algebra,
so that every
compact pseudo-Riemannian nilmanifold is an Einstein manifold.
For solvable Lie algebras that are not nilpotent, the Killing
form does not vanish in general.
By a metric Lie algebra we mean a Lie algebra with an invariant
scalar product.
We find the following constraints on the dimension of $\frg$:

\begin{thm}\label{thm:einstein_dim}
Let $(\frg,\met)$ be a solvable
metric Lie algebra with Einstein scalar product of Witt index $s$.
Let $\frn$ be the nilradical of $\frg$.
If $\frg$ is not nilpotent, then $\dim\frg\geq 6$,
$\dim\frn\geq 5$,
and $s\geq 2$.
Example \ref{ex:kappa=0} shows that these estimates are sharp.
\end{thm}

The existence of compact pseudo-Riemannian quotients of a
solvable Lie group $G$ requires the existence of a lattice
$\Gamma$ in $G$, and it is not clear whether such a lattice
can exist under the condition that $\met$ is Einstein.
We can rule this out for dimensions below $8$:

\begin{thm}\label{thm:no_einstein_dim6}
Let $M$ be a compact pseudo-Riemannian Einstein solvmanifold of
dimension less or equal to $7$.
Then $M$ is a nilmanifold.
\end{thm}

The proof of this theorem uses the Gelfond-Schneider Theorem from
the theory of transcendental numbers to show that no lattice can
exist.
However, application of the Gelfond-Schneider Theorem
seems limited to low
dimensions.
For arbitrary dimensions, the statement would hold if Schanuel's
Conjecture on the algebraic independence of certain complex
numbers is true. We therefore conjecture:

\begin{conjecture}\label{conj:no_einstein}
Every compact pseudo-Riemannian Einstein solvmanifold
is a nilmanifold.
\end{conjecture}

In Section \ref{sec:semisimple} we consider
pseudo-Riemannian homogeneous Einstein manifolds $M=G/H$
of finite volume, where $G$ is a semisimple Lie group.
We may assume that $G$ is simply connected, so $G=K\times S$, where
$K$ is compact semisimple and $S$ is semisimple without compact
factors.

\begin{thm}\label{thm:Einstein_semisimple}
Let $M$ be a pseudo-Riemannian homogeneous Einstein manifold
of finite volume
and $G$ a semisimple, connected and simply connected Lie group
acting transitively and almost effectively by isometries on $M$.
Identify $M$ with $G/H$ for a closed subgroup $H$ of $G$.
Let $G=K\times S$ be the decomposition of $G$ into a compact
semisimple factor $K$ and a semisimple factor $S$ without
compact simple factors.
Then:
\begin{enumerate}
\item
The identity component $H^\circ$ of $H$ is contained in $K$.
\item
$M=\hat{M}/\Gamma$, where $\hat{M}$ is a pseudo-Riemannian product of
Einstein manifolds
\[
\hat{M}=M_K\times S,
\]
with $M_K=K/(H\cap K)$,
and the Einstein metric on $S$ is a multiple of the Killing form on $S$.
\item
$\Gamma$ is the graph of a homomorphism
from a lattice in $S$ into $K$.
\end{enumerate}
\end{thm}

A classification of homogeneous Einstein manifolds of finite volume
with transitive action by a semisimple group thus reduces to
classifying Einstein quotients of compact semisimple
Lie groups,
classifying lattices in semisimple Lie groups without compact
factors together with their embeddings in products
of compact and non-compact semisimple groups whose right-translations
preserve the metric.
The first problem is a hard problem even if $\g_M$ is
a Riemannian metric, and several different approaches have been used in this case to construct examples and find non-existence
proofs, see for example D'Atri and Ziller \cite{DZ},
Wang and Ziller \cite{WZ}, or B\"ohm et al.~\cite{boehm, BK, BWZ}.

In the proof of Theorem \ref{thm:Einstein_semisimple},
the finite volume of $M$ 
requires a bi-invariant Einstein metric on $S$ which is then
necessarily given as a multiple of the Killing form.
However, there exist many other pseudo-Riemannian Einstein metrics
for most semisimple groups.
Some low-dimensional examples were constructed by Gibbons,
L\"u and Pope \cite{GLP}.
Chen and Liang \cite{CL} studied the existence of pseudo-Riemannian
Einstein metrics on three- and four-dimensional Lie groups. 
The moduli spaces of Einstein connections for simple Lie groups
were studied systematically by Derdzinski and Gal \cite{DG}.

%%%%
%\subsection*{Notations and conventions}
%
%Recall that the central series of $\frg$ is defined by
%$\Cen^0\frg=\frg$ and $\Cen^{j+1}\frg=[\frg,\Cen^j\frg]$,
%$j\geq0$,
%and that the ascending central series of $\frg$ is
%defined by
%$\Cen_0\frg=\zsp$ and $\Cen_{j+1}\frg=\{x\in\frg\mid [x,\frg]\subset\Cen_j\frg\}$ for $j\geq 0$.
%If $\Cen^k\frg=\zsp$, we say that $\frg$ is
%\emph{$k$-step nilpotent}.
%We say that $\frg$ is \emph{precisely $k$-step nilpotent} if
%in addition $\Cen^{k-1}\frg\neq\zsp$ holds.
%
%For a Lie algebra $\frf\subset\gl(V)$ of endomorphisms on a
%vector space $V$, we use the following notation
%\[
%\ker\frf = \bigcap_{x\in\frf}\ker x,
%\quad
%\im\frf = \sum_{x\in\frf}\im x.
%\]

\subsection*{Acknowledgments}
We would like to thank Anna Fino and Oliver Baues for
helpful discussions, and we would also like to thank
Michel Waldschmidt and Arturas Dubickas for useful
comments on transcendental numbers.
We would also like to thank Ines Kath for pointing out
a mistake in a previous version of Theorem \ref{thm:Einstein_semisimple}.

% !TEX root = einsteinCompactSolvable.tex

%%%%%%%%%%%%%%%%%%%%%%%
\section{Invariant scalar products}
%%%%%%%%%%%%%%%%%%%%%%%
\label{sec:scalar}

Let $\frg$ be a Lie algebra of finite dimension $n$
with a symmetric bilinear form $\met$.
The form $\met$ is called
\emph{invariant} if
\begin{equation}
\langle [x,y_1],y_2\rangle = - \langle y_1,[x,y_2]\rangle
\end{equation}
for all $x,y_1,y_2\in\frg$.
The \emph{metric radical}
\[
\frg^\perp = \{x\in\frg \mid \langle x,y\rangle=0 \text{ for all } y\in\frg\}
\]
is an ideal in $\frg$.

An element $x\in\frg$ is called \emph{isotropic} if
$\langle x,x\rangle=0$.
A subspace $\frv\subset\frg$ is called \emph{totally isotropic}
if $\frv\subset\frv^\perp$.

If $\met$ is non-degenerate, that is, if $\frg^\perp=\zsp$,
then the maximal dimension $s$ of a totally isotropic subspace of
$\frg$ is called the \emph{Witt index} (or simply \emph{index})
of $\met$, and the pair $(n-s,s)$ is called the \emph{signature}
of $\met$.
In this case we call $\met$ a \emph{scalar product} on $\frg$,
and the pair $(\frg,\met)$
is called a \emph{metric Lie algebra}.
A totally isotropic subspace $\frv$ of dimension $s$ is also called
a \emph{maximally isotropic} subspace.

\begin{example}
Consider $\RR^n$ with a scalar product $\met$ represented by the
matrix
\[
\begin{pmatrix}
\I_{n-s} & 0 \\
0 & -\I_s
\end{pmatrix},
\]
where $s\leq n-s$. Then $\met$ has index $s$.
The space $\RR^n$ together with $\met$ becomes an abelian
metric Lie algebra denoted by $\ab^n_s$.
\end{example}

Given a totally isotropic subspace $\fru$ of $\frg$,
we can find a non-degenerate subspace $\frw$ such that
$\fru^\perp=\frw\oplus\fru$, and a totally isotropic subspace $\frv \subset \frw^\perp$ 
such that $\frv$ is dually paired with $\fru$ by $\met$. 
The resulting decomposition
\[
\frg = \frv \oplus\frw\oplus\fru 
\]
is called a \emph{Witt decomposition} of $\frg$ relative to
$\fru$.
If $\{v_1,\ldots,v_k\}$ is a basis for $\frv$,
$\{v_1^*,\ldots,v_k^*\}$ its dual basis for $\fru$,
and $\{w_1,\ldots,w_{n-2k}\}$ an orthonormal basis for $\frw$,
then we call the basis
$\{v_1,\ldots,v_k,w_1,\ldots,w_{n-2k},v_1^*,\ldots,v_k^*\}$
a \emph{Witt basis} for $\frg$.

For our purposes, a more general notion than that of invariance
of $\met$ is required:

\begin{definition}\label{def:nilinvariant}
A symmetric bilinear form $\met$ on $\frg$ is called
\emph{nil-invariant} if for all $x_1,x_2\in\frg$,
\begin{equation}
\langle \varphi_{\rm n}x_1,x_2\rangle = - \langle x_1,\varphi_{\rm n}x_2\rangle,
\label{eq:nilinvariant}
\end{equation}
where $\varphi$ is an element of the
Lie algebra of the Zariski closure of $\Ad(G)$ in $\Aut(\frg)$,
and $\varphi_{\rm n}$ denotes the nilpotent part of its
additive Jordan decomposition, where $G$ is any connected Lie group with
Lie algebra $\frg$.
\end{definition}

For nil-invariant $\met$, \eqref{eq:nilinvariant} holds for all
$\varphi_{\rm n}=\ad(y)_{\rm n}$ with $y\in\frg$.
In particular, $\met$ is $\ad(\frn)$-invariant, where
$\frn$ is the nilradical of $\frg$.
It was shown by Baues, Globke and Zeghib \cite{BGZ} that
every nil-invariant symmetric bilinear form is invariant
for a Lie algebra $\frg=\frs\ltimes\frr$, where $\frs$ is
semisimple without factors of compact type and $\frr$ is the
solvable radical.
On the other hand, on a semisimple Lie algebra $\frk$ of compact
type, every bilinear form is nil-invariant.

\begin{remark}\label{rem:nilinvariant}
The notion of nil-invariance is relevant for homogeneous spaces
$M=G/H$ as defined in the introduction.
Here, $G$ preserves a
probability measure on $M$ since it acts isometrically.
Therefore, a theorem by Mostow \cite[Lemma 3.1]{mostow3}
implies that $\met$ is preserved by the Zariski closure  $\ac{\Ad(H)}$
of $\Ad(H)$ in $\Aut(\frg)$, and thus by
all unipotent parts of $\ac{\Ad(G)}$.
But this  means that $\met$ is nil-invariant
(see Sections 2 and 6 of Baues and Globke \cite{BG} for
further details).
\end{remark}

\begin{definition}
If a symmetric bilinear form $\met$ on $\frg$ comes from 
the pullback of an Einstein metric $\g_M$ on a homogeneous space
$M=G/H$, then we call $\met$ an \emph{Einstein form}.
If in addition $\met$ is non-degenerate, we call it an
\emph{Einstein scalar product}.
\end{definition}

% !TEX root = einsteinCompactSolvable.tex

%%%%%%%%%%%%%%%%%%%%%%%
\section{Reductions of solvable metric Lie algebras}
%%%%%%%%%%%%%%%%%%%%%%%
\label{sec:reduction}

Let $\frg$ be a finite-dimensional solvable Lie algebra and
$\met$ an invariant scalar product on $\frg$.
In this section, we review some of the properties for such
metric Lie algebras $(\frg,\met)$. 
Proofs can be found in Baues and Globke
\cite[Sections 4 and 5]{BG}.

%%%
\subsection{Reduction by a totally isotropic ideal}

Let $\frj \subset\zen(\frg)$ be a central totally 
isotropic ideal in $\frg$.
Then $\frj^\perp$ is an ideal in $\frg$.
In particular, we can consider the quotient Lie algebra
\[
\bar \frg = \frj^\perp / \frj.
\]
Since $\frj$ is totally isotropic, the restriction of $\met$ to
$\frj^{\perp}$ induces a symmetric bilinear form on $\bar{\frg}$,
also denoted by $\met$. The induced form on $\bar{\frg}$ is
non-degenerate since $\met$ is non-degenerate on $\frg$ and
the metric radical of its restriction to $\frj^\perp$ is precisely
$\frj$.

\begin{definition}
The metric Lie algebra $(\bar \frg, \langle\cdot,\cdot\rangle)$ 
is called the \emph{reduction} of $(\frg, \langle\cdot,\cdot\rangle)$ by the totally isotropic ideal $\frj$.
\end{definition}

Let
\begin{equation} \label{eq:decompositionW}
\frg = \fra\oplus\frj^\perp = \fra \oplus\frw\oplus\frj
\end{equation}
be a Witt decomposition relative to $\frj$, where
$\fra$ is a dual space to $\frj$ with respect to $\met$ and
$\frw$ is a non-degenerate subspace orthogonal to $\fra$ and $\frj$.

For all $x \in \frg$, we write $x = x_{\fra} + x_{\frw} + x_{\frj}$ with respect to \eqref{eq:decompositionW}.
In the following we shall frequently identify ${\frw}$ with the underlying vector space of $\bar \frg$. Thus for $x \in \frj^\perp$, the projection $\bar{x}$ of $x$ to $\bar{\frg}$ may also be identified with the element $x_{\frw} \in\frw$. Similarly, 
$[ \bar{x}, \bar{y} ]_{\bar \frg} = ([x,y])_{\frw}$ for $x,y \in  \frj^\perp$ is the Lie bracket in $\bar \frg$. 
The Lie product in $\frg$
thus gives rise to the following equations.

For all $x,y \in \frj^\perp$, 
\begin{equation}
\label{eq:defomega}
[x,y] =  [\bar{x},\bar{y}]_{\bar{\frg}} + \omega(\bar{x}, \bar{y}),
\end{equation} 
where $\omega\in\Z^2(\bar \frg, \frj)$ is a $2$-cocycle.

For all $a \in \fra$, $x \in \frj^{\perp}$,
\begin{equation}
[a , x]  = \delta_a \,\bar{x} + \xi_{a}(\bar{x}),
\label{eq:A_action}
\end{equation}
where $\xi_{a}: \bar \frg \to \frj$ is a linear map, 
and $\delta_a $ is the derivation of $\bar{\frg}$ 
induced by $\ad(a)$.

In a split situation, the maps $\xi_{a}$ vanish: 
\begin{lem}\label{lem:An1}
Assume that $[\fra, \fra] = \zsp$ (that is, $\fra$ is an abelian 
subalgebra). Then $[\fra, \frg]$ is contained in $\fra^{\perp}$. 
In particular, $\xi_{a} = 0$ for all $a \in \fra$. 
\end{lem}

%\oliver{Das Lemma  funktioniert strukturell auch unter Bedingungen, wie $[\fra, \fra] \subset \frn^{\perp}$, aber es ist nicht klar, wann das n\"utzlich wird.}

It turns out the derivation $\delta_a $ and the extension cocycle $\omega$ determine each other: 
\begin{prop}  \label{prop:A_and_O}
For all $x,y\in \frj^{\perp}$, $a\in\fra$, we have 
\begin{equation}
\label{eq:A_and_O}
   \langle [a,x],y\rangle
  %\langle\bar{A}\,  \bar X,  Y\rangle
= \langle \omega(\bar{x}, \bar{y}), a \rangle .
\end{equation}
\end{prop}

Every non-abelian solvable metric Lie algebra $(\frg, \met)$ with
invariant scalar product $\met$ admits a non-trivial
totally isotropic  and central ideal $\frj$,
see Proposition \ref{prop:zgzn}.
Therefore, $(\frg, \met)$ reduces 
to a metric Lie algebra  $(\bar \frg, \met)$ of lower dimension.
This reduction is still possible for abelian Lie algebras as
long as $\met$ is indefinite.
Iterating this procedure we obtain:

\begin{prop} \label{prop:completereduction}
Let $(\frg, \met)$ be a solvable metric 
Lie algebra with  nil-invariant non-degenerate symmetric bilinear form $\met$. After a finite sequence of successive reductions 
with respect to one-dim\-en\-sion\-al totally isotropic and central ideals,  
$(\frg, \met)$ reduces to an abelian metric Lie algebra 
with a definite scalar product $\met$.
\end{prop}

If a reduction $(\bar{\frg},\met)$ has a definite
scalar product, then it cannot be reduced further.
In this case we call it a \emph{complete reduction}.
From Proposition \ref{prop:completereduction} we immediately
obtain:

\begin{cor}\label{cor:definite_base}
If $\dim\frg=n$ and the Witt index of $\met$ is $s$, then
a complete reduction of $(\frg,\met)$ is isometric to
$(\ab^{n-2s}_0,\met_0)$,
where $\met_0$ denotes a definite
scalar product on the underlying vector space of $\ab^{n-2s}_0$.
\end{cor}

%%%
\subsection{The characteristic ideal  $\boldsymbol{\zen(\frn)\cap[\frg,\frn]}$.} \label{sec:j0}
Let $\frn$ denote the nilradical of $\frg$, and
\begin{equation}
\frj_0 = \zen(\frn) \cap [\frg,\frn].
\label{eq:j}
\end{equation}
Occasionally we will also write $\frj_0(\frg)$ to distinguish
between the ideals of different Lie algebras.
The ideal $\frj_0$ is totally isotropic,
and $\frj_0^\perp$ is an ideal in $\frg$.

A fundamental property is:

\begin{prop}\label{prop:j0}
$\frg$ is abelian if and only if\/ $\frj_0 = \zen(\frn) \cap [\frg,\frn] ={\bf 0}$.
\end{prop}

\begin{prop}\label{prop:zgzn}
If $\frg$ is non-abelian, then
$\frj_0$ contains a non-zero totally isotropic
central ideal of $\frg$.
\end{prop}

Invariance of $\met$ then implies
$[\frg,\frg]\subset\frj_0^\perp$.

\subsection{Double extensions}\label{subsec:double_extension}

The process of double extension reverses reduction.
It was first introduced by Medina and Revoy \cite{MR}.
Given an arbitrary Lie algebra $\bar{\frg}$ with an invariant
scalar product $\met$, a Lie algebra $\fra$ together with a
representation $\delta$ of $\fra$ by
skew-symmetric derivations $\delta_a \in\der(\bar{\frg})$,
and a two-cocycle $\omega\in\mathrm{Z}^2(\bar{\frg},\fra^*)$
defined by
\[
\omega({x},{y})(a) = \langle\delta_a \, x,y\rangle,
\]
where $a\in\fra$, ${x},{y}\in\bar{\frg}$,
we can define a Lie algebra product on the vector space
$\frg=\fra\oplus\bar{\frg}\oplus\fra^*$ by
\[
[a_1+{x}_1+\alpha_1,a_2+{x}_2+\alpha_2] 
=
[a_1,a_2]_{\fra}
\ +\
[{x}_1,{x}_2]_{\bar{\frg}}+\delta_{a_1}{x}_2-\delta_{a_2}{x}_1
\ +\
\omega({x}_1,{x}_2),
\]
where $a_i\in\fra$, ${x}_i\in\bar{\frg}$ and $\alpha_i\in\fra^*$.

We can extend the invariant inner product $\met$ on $\bar{\frg}$
to the Lie algebra $\frg$ via
\[
\langle a_1+{x}_1+\alpha_1,a_2+{x}_2+\alpha_2\rangle
=
\langle {x}_1,{x}_2\rangle
+\alpha_1(a_2)+\alpha_2(a_1).
\]

\begin{definition}
We call $\frg$ the \emph{double extension} of $\bar{\frg}$
by $(\fra,\delta)$.
\end{definition}

Observe that $\frj=\fra^*$ is a totally isotropic ideal contained
in $\zen(\frg)$, and that $\frj^\perp=\bar{\frg}\oplus\frj$.
So we can apply reduction by $\frj$ as defined above, and
thereby recover the metric Lie algebra $(\bar{\frg},\met)$
from $(\frg,\met)$.

% !TEX root = einsteinCompactSolvable.tex

%%%%%%%%%%%%%%%%%%%%%%%
\section{Nil-invariant Einstein metrics}
%%%%%%%%%%%%%%%%%%%%%%%
\label{sec:nilinvariantEinstein}

The Ricci tensor $\Ric$ for a bi-invariant pseudo-Riemannian
metric on a Lie group $G$ is determined by
\begin{equation}
\Ric|_{\frg}=-\frac{1}{4}\kappa,
\label{eq:Ric}
\end{equation}
where $\kappa$ is the Killing form of $G$
(O'Neill \cite[Corollary 11.10]{oneill}).
Recall that $(G,\g)$ is an \emph{Einstein manifold} if
\begin{equation}
\Ric=\lambda\g
\label{eq:Einstein}
\end{equation}
for some constant $\lambda$.

If $G$ is not semisimple, then $\kappa$ is degenerate, and so the
only way to satisfy both \eqref{eq:Ric} and \eqref{eq:Einstein} is
with $\lambda=0$.
This in turn implies that $\kappa=0$, which means $G$ does not have
semisimple factors.
We have thus shown:

{
\renewcommand{\thethm}{\ref{prop:Einstein_solvable}}
\begin{prop}
If $\g$ is a bi-invariant Einstein metric on $G$,
then
\begin{enumerate}
\item[(a)]
either $G$ is semisimple and $\g$ is a non-zero multiple of the Killing
form,
\item[(b)]
or $G$ is solvable and its Killing form vanishes.
\end{enumerate}
\end{prop}
\addtocounter{thm}{-1}
}

We can apply similar arguments to a pseudo-Riemannian
homogeneous Einstein manifold $M$ of finite volume.
Let $G$ be a connected Lie group that acts transitively
and isometrically on $M$. Further, assume that $G$ acts almost
effectively.
Then $M=G/H$ for a closed subgroup of $H$ whose connected
component does not contain a non-trivial normal subgroup of $G$.

Let $\frg$ and $\frh$ denote the respective Lie algebras of $G$
and $H$.
By pulling back the pseudo-Riemannian metric from $M$ to $G$,
we obtain a symmetric bilinear tensor $\g_G$ on $G$ which induces
a symmetric bilinear form $\met$ on $\frg$.
The metric radical $\frg^\perp$ of $\met$ coincides with $\frh$,
and $\met$ is $\Ad(H)$-invariant.
Recall from Remark \ref{rem:nilinvariant} that this implies that
$\met$ is nil-invariant.

Assume $G$ to be simply connected and
let
\[
G=(K\times S)\ltimes R
\]
be the Levi decomposition of $G$,
where $K$ is compact semisimple, $S$ is semisimple without
compact factors, and $R$ is the solvable radical of $G$.
Let $\frk$, $\frs$ and $\frr$ denote
their respective Lie algebras.

\begin{thm}[Baues, Globke \& Zeghib \cite{BGZ}]\label{thm:invariance}
Let $\frg$ be a finite-dimensional Lie algebra with nil-invariant symmetric bilinear
form $\met$.
Then:
\begin{enumerate}
\item
The form $\met$ on $\frg$ is invariant under
$\ad_{\frg}(\frs\ltimes\frr)$.
\item
The restriction $\met_{\frs\ltimes\frr}$
of $\met$ to $\frs\ltimes\frr$
is invariant under the adjoint action of $\frg$. 
In particular, $\met_{\frs\ltimes\frr}$ is invariant.
\end{enumerate}
\end{thm}

Since the stabilizer algebra $\frh$ coincides with the
metric radical $\frg^\perp$, invariance of a scalar product
implies that $\frh$ is an ideal in $\frg$.
Together with the assumption that $G$ acts almost effectively,
this means $\frh$ is trivial.
We have the following consequence:

\begin{cor}\label{cor:no_compact}
If $K=\one$, then the stabilizer $H$ is a lattice
in $G$.
\end{cor}

%\begin{lem}\label{lem:Rnondeg}
%With $G,R,\frg,\frr$ as above,
%assume that the restriction $\met_\frr$ of $\met$ to $\frr$ is
%non-degenerate.
%Then the Ricci tensor of $R$ is given by
%$\Ric_R|_\frr=-\frac{1}{4}\kappa_\frr$, where $\kappa_\frr$ is the
%Killing form of $\frr$.
%\end{lem}
%\begin{proof}
%Since $\met$ is nil-invariant, the restriction $\met_\frr$ is
%invariant (Theorem \ref{thm:invariance}). So the induced metric
%tensor $\g_R$ on $R$ is bi-invariant and thus satisfies
%$\nabla_X Y=\frac{1}{2}[X,Y]$ for all $X,Y\in\frr$ (identified with
%left-invariant vector fields).
%In particular, $\nabla_X Y\in\frr$, so the normal component of
%$\nabla_X Y$ is $0$. Hence $R$ is totally geodesic and it follows
%by the usual computations that
%$\Ric_R|_{\frr}=-\frac{1}{4}\kappa_\frr$.
%\end{proof}

\begin{cor}\label{cor:alternative1}
Let $M=G/H$ be a pseudo-Riemannian Einstein manifold of finite
volume, with a connected Lie group $G$ acting isometrically and
almost effectively.
If a Levi subgroup of $G$ has no compact factors, then:
\begin{enumerate}
\item
$H$ is a lattice in $G$.
\item
The pseudo-Riemannian Einstein metric on $M$ pulls back to a
bi-invariant
pseudo-Riemannian Einstein metric on $G$.
\item
$G$ is either solvable or semisimple without compact factors.
\end{enumerate}
\end{cor}
\begin{proof}
(1) and (2) follow from Theorem \ref{thm:invariance}
and Corollary \ref{cor:no_compact}.
Then (3) follows from Proposition \ref{prop:Einstein_solvable}.
\end{proof}

%\begin{cor}\label{cor:alternative2}
%Let $M=G/\Gamma$ be a pseudo-Riemannian Einstein manifold of finite
%volume, with a connected Lie group $G$ acting isometrically and
%almost effectively, and $\Gamma$ a lattice in $G$.
%If the Einstein metric $\g_M$ on $M$ pulls back to a bi-invariant
%symmetric bilinear form $\g$ on $G$, then
%$G$ is either solvable or semisimple.
%\end{cor}
%\begin{proof}
%Since the stabilizer is a lattice, $G$ and $M$ are locally
%isometric. Hence $\g_M$ pulls back to a
%bi-invariant Einstein metric $\g$ on $G$.
%The statement now follows from Proposition \ref{prop:Einstein_solvable}.
%\end{proof}

Suppose now that $G$ is solvable.
The metric radical of $\kappa$ contains the nilradical of the Lie algebra
$\frg$ of $G$.
So if $G$ is nilpotent, then $\kappa=0$.
However, as Example \ref{ex:kappa=0} below shows,
there are $\frg$ which are solvable but not nilpotent
and also have $\kappa=0$.

\begin{example}\label{ex:kappa=0}
Let $\{b,x_1,x_2,y\}$ be a Witt basis of $\ab^4_1$ with isotropic
vectors $b,y$, $\langle b,y\rangle=1$, $\langle x_i,x_j\rangle=\delta_{ij}$ and $x_i\perp\Span\{b,y\}$. 

In this basis, define a semisimple skew derivation by
\[
\delta_a
=
\begin{pmatrix}
1 & 0 & 0 & 0 \\
0 & 0 & -1 & 0 \\
0 & 1 & 0 & 0 \\
0 & 0 & 0 & -1
\end{pmatrix}.
\]
Double extension of $\ab_1^4$ by $\Span\{a\}$, where $a$ acts on
$\ab_1^4$ by $\delta_a$, yields a six-dimensional metric
Lie algebra $\frg$ with invariant scalar product of index
$2$.
The generators of $\frg$ are $a,b,x_1,x_2,y,z$, where $z$ spans the
center of $\frg$.

The Killing form $\kappa$ of $\frg$ is zero, since
\[
\tr(\ad(a)^2)=\tr(\delta_a^2)=0,
\]
and all other adjoint operators are nilpotent.
A Lie group $G$ with Lie algebra $\frg$ has a bi-invariant Einstein
metric given by left-invariant continuation of $\met$.
However, we will see in Theorem \ref{thm:no_einstein_dim6} that no
such $G$ admits a lattice, so this example does not give rise
to a compact homogeneous Einstein manifold.
\end{example}

\section{Solvable isometry groups}
%%%%%%%%%%%%%%%%%%%%%%%
\label{sec:solvable}

Let $M$ be a compact pseudo-Riemannian Einstein solvmanifold,
that is, it has a transitive action by a connected
solvable Lie group $G$ of isometries.
It was shown by Baues and Globke \cite[Theorems A, B]{BG} that
$M=G/\Gamma$, where $\Gamma$ is a lattice in $G$ and the metric
$\g_M$ on $M$ pulls back to a bi-invariant pseudo-Riemannian
metric $\g$ on $G$.
Moreover, $G=\Iso(M)^\circ$ by \cite[Corollary D]{BG}.

Due to the local freeness of the $G$-action,
$\g$ is an Einstein metric on $G$.
By Proposition \ref{prop:Einstein_solvable}, the Killing form
of $G$ is $\kappa=0$.

We determine conditions for a solvable Lie algebra $\frg$
to admit an invariant Einstein scalar product $\met$
(that is, to satisfy $\kappa=0$) and for $G$ to have a lattice
subgroup at the same time.

%%%
\subsection{Einstein scalar products}

Let $\frn$ denote the nilradical of $\frg$. There is a
vector space decomposition
\begin{equation}
\frg = \fra\oplus\frn
\label{eq:g=a+n}
\end{equation}
for some subspace $\fra$ of $\frg$ whose elements act by
non-nilpotent derivations on $\frn$.
For $x\in\frn$, the derivation $\ad(x)$ is nilpotent,
so $\kappa(x,y)=0$ for any $y\in\frg$ is automatically satisfied.
So the condition for $\met$ to be Einstein is
\[
\kappa(a,b)=\tr(\ad(a)\ad(b))=0
\quad\text{ for all } a,b\in\fra.
\]
By polarization it is enough to consider the corresponding
quadratic equations
\begin{equation}
\kappa(a,a)=\tr(\ad(a)^2)=0
\quad\text{ for all } a\in\fra.
\label{eq:kappa0}
\end{equation}
Given an element $a\in\fra$, let
\begin{equation}
\lambda_1,\ldots,\lambda_k\in\RR,\
\zeta_1,\ldots,\zeta_m,\bar{\zeta}_1,\ldots,\bar{\zeta}_m\in\CC\backslash\RR
\label{eq:eigenvals}
\end{equation}
denote the eigenvalues of $\ad(a)$, and set $\alpha_i=\Re(\zeta_i)$,
$\beta_i=\Im(\zeta_i)$ for each $i$.
Then condition \eqref{eq:kappa0} translates into
\begin{equation}
\lambda_1^2+\ldots+\lambda_k^2
+2\alpha_1^2+\ldots+2\alpha_m^2
-2\beta_1^2-\ldots-2\beta_m^2=0.
\label{eq:kappa_square}
\end{equation}
The second condition is that $\ad(x)$ is a skew-symmetric
with respect to $\met$,
\begin{equation}
\langle \ad(x)y,z\rangle = -\langle y,\ad(x)z\rangle
\label{eq:ad_skew}
\end{equation}
for all $x,y,z\in\frg$.
This means $\ad(\frg)$ is a solvable subalgebra of
$\sso(p,q)\cong\sso(\frg,\met)$ for certain $0\leq q\leq p\leq 0$,
and as such it is contained in a maximal solvable subalgebra
$\frb$ of $\sso(p,q)$.
The maximal solvable subalgebras of $\sso(p,q)$ were described
by Patera, Winternitz and Zassenhaus \cite{PWZ}
(assume $p\geq q>0$):
\begin{enumerate}
\item
The first possibility is that $\frb$ is a maximally compact
Cartan subalgebra.
Every element $X\in\frb$ can be represented as
\[
X = \sum_{i=1}^{\lfloor\frac{p}{2}\rfloor}\xi_i (E_{2i-1,2i}-E_{2i,2i-1})
+\sum_{j=1}^{\lfloor\frac{q}{2}\rfloor}\mu_j (E_{p+2j-1,p+2j}-E_{p+2j,p+2j-1})
\]
where $E_{i,j}$ denotes the matrix with entry $1$ in row $i$,
column $j$, and all other entries $0$.
A direct calculation shows that $\tr(X^2)<0$ whenever
$X\neq 0$.
By \eqref{eq:kappa0}, $\ad(\fra)$ cannot be contained in such an
algebra $\frb$.
\item
If $\frb$ is not maximally compact, then the elements
$X\in\frb$ can be simultaneously represented recursively as
follows:
\begin{equation}
X
=
\begin{pmatrix}
A & * & * \\
0 & X_1 & * \\
0 & 0 & -A^\top
\end{pmatrix},
\label{eq:PWZ}
\end{equation}
where
\begin{enumerate}
\item
either $A=\lambda\in\RR$ and $X_1\in\sso(p-1,q-1)$ (up to conjugation),
\item
or $A=\left(\begin{smallmatrix} \alpha & -\beta\\ \beta & \alpha\end{smallmatrix}\right)$
and $X_1\in\sso(p-2,q-2)$ (up to conjugation).
\end{enumerate}
In either case, $X_1$ is (conjugate to) an element of a maximal
solvable subalgebra of the respective lower-dimensional
pseudo-orthogonal algebra,
and thus has a re\-presentation of the same form \eqref{eq:PWZ}.
For $\tr(X^2)$ we obtain
\begin{equation}
\tr(X^2) = 2\tr(A)^2 + \tr(X_1^2)
=\left\{
\begin{array}{ll}
2\lambda^2 +\tr(X_1^2) & \text{in case (a)}, \\
4\alpha^2-4\beta^2 + \tr(X_1^2) & \text{in case (b)}.
\end{array}\right.
\label{eq:recursion}
\end{equation}
This can be iterated for $\tr(X_1^2)$ unless the block $X_1$
is contained in $\sso(r)$ for some $r\geq 0$.
Should this be the case, $X_1$ is (conjugate to) a block diagonal
matrix with eigenvalues in $\i \RR$.
\end{enumerate}

\begin{remark}
Iterating \eqref{eq:recursion} for $X_1$ yields again the form
\eqref{eq:kappa_square},
where the $\lambda_i^2$ that are contributions from the cases (a)
and the $\alpha_j^2,\beta_j^2$ that are contributions from the
cases (b) in the recursion appear with even multiplicity.
Only eigenvalues in $\i\RR$ coming from the block
$X_1\in\sso(r)$ in the last step of the recursion may appear with
odd multiplicity.
\end{remark}

{
\renewcommand{\thethm}{\ref{thm:einstein_dim}}
\begin{thm}
Let $(\frg,\met)$ be a solvable
metric Lie algebra with Einstein scalar product of Witt index $s$.
Let $\frn$ be the nilradical of $\frg$.
If $\frg$ is not nilpotent, then $\dim\frg\geq 6$,
$\dim\frn\geq 5$,
and $s\geq 2$.
Example \ref{ex:kappa=0} shows that these estimates are sharp.
\end{thm}
\addtocounter{thm}{-1}
}
\begin{proof}
As $\frg$ is not nilpotent, there exists $a\in\frg\backslash\frn$
that acts on $\frg$ with non-vanishing semisimple part
$\sigma_a=\ad(a)_{\rm ss}$ of the Jordan decomposition of
$\ad(a)$.
Set $W_1=\im\sigma_a$ and $W_0=\ker\sigma_a$.
Then $\frg=W_1\oplus W_0$,
$W_0\perp W_1$ by invariance of $\met$,
and $\dim W_0>0$ since $\sigma_a(a)=0$.
Note that $W_1\subset[\frg,\frg]\subset\frn$ since
$\im\sigma_a\subseteq\im\ad(a)$.
Moreover, $\dim W_1$ is even since all non-zero eigenvalues of
$\ad(a)$ come in pairs of either complex conjugates or
positives and negatives.
We proceed in several steps:

%\begin{enumerate}
%\item
First, we show that $\dim W_1\geq 4$.
By \eqref{eq:kappa_square}, in order to satisfy
\[
\tr(\ad(a)^2)=\tr(\sigma_a^2)=\tr((\sigma_a|_{W_1})^2)=0,
\]
we need at least one
non-real eigenvalue $\zeta=\alpha+\i\beta$ of $\sigma_a|_{W_1}$.
Use the representation \eqref{eq:PWZ} to deduce the following:
\begin{itemize}
\item
If $\alpha\neq0$, then $-\zeta$ is also an
eigenvalue of $\ad(a)$.
But then we have four distinct eigenvalues
$\zeta,\bar{\zeta},-\zeta,-\bar{\zeta}$.
This is only possible if $\dim W_1\geq 4$.
\item
If $\alpha=0$, then either
\[
2\lambda^2-2\beta^2 = 0
\]
for an eigenvalue $\lambda\in\RR$ whose negative $-\lambda$ must
also be an eigenvalue,
or
\[
2\alpha_1^2-2\beta_1^2-2\beta^2 = 0
\]
for an eigenvalue $\alpha_1+\i\beta_1\in\CC\backslash(\RR\cup\i\RR)$
(the existence of more eigenvalues already requires
$\dim W_1>4$).
In both cases, we have at least four distinct eigenvalues,
so $\dim W_1\geq 4$.
\end{itemize}
If $W_1$ is a definite subspace, then $\sigma_a|_{W_1}\in\sso(4)$,
which contradicts any of the eigenvalue combinations above.
So $W_1$ contains a totally isotropic line $L$.

%\item
Next, we show that the index of $\met$ is at least $2$.
The totally isotropic ideal $\frj_0=\zen(\frn)\cap[\frg,\frn]$
from \eqref{eq:j} is non-zero
since $\frg$ is not abelian (Proposition \ref{prop:j0}).
There exists a non-trivial ideal $\fri\subset\frj_0\cap\zen(\frg)$ (Proposition \ref{prop:zgzn}),
so that $\fri\subset W_0$.
Then also $\fri\perp[\frg,\frg]$ due to invariance of $\met$,
so in particular $\fri\perp W_1$.
Therefore, $\fri\oplus L$ is a totally isotropic subspace of
dimension $\geq 2$ in $\frg$, which means that the index of $\met$
is $\geq 2$.

%\item
We show that $\dim\frg\geq 6$.
From the previous steps we know that $\dim W_1\geq 4$ and
$\dim\fri\geq 1$. Hence
\[
\dim\frn\ \geq\ \dim W_1+\dim\fri\ \geq\ 4+1=5. 
\]
Therefore, since $a\nin\frn$,
\[
\dim\frg\ \geq\ \dim\RR a + \dim\frn
\ \geq\ 1+5=6.
\]
This concludes the proof of the theorem.
%Since $a\perp W_1$,
%there exists $a^*\in W_0$ with $\langle a,a^*\rangle=1$.
%Suppose $\dim\frg=5$ (so that $\dim W_1=4$).
%Then we may assume
%$a^*=a$, so that $W_0$ is non-degenerate.
%Since $W_0\perp W_1$, $W_1$ is also non-degenerate.
%So for every $x\in W_1$, we have $y=[a,x]\in W_1$ and thus
%there exists $y^*\in W_1$ such that
%\[
%1=\langle y,y^* \rangle
%=\langle [a,x],y^* \rangle
%=\langle a,[x,y^*] \rangle.
%\]
%Then $[x,y^*]$ must have a non-zero $\RR a$-projection,
%hence $\ad([x,y^*])_{\rm ss}\neq0$.
%But this contradicts
%$[x,y^*]\in[\frg,\frg]\subset\frn$.
%Therefore, $a^*\nin\RR a\oplus W_1$, and $\dim\frg\geq 6$
%follows.
%\end{enumerate}
\end{proof}

\begin{cor}\label{cor:Lorentz_Einstein}
Every Lorentzian Einstein metric Lie algebra $(\frg,\met)$ is abelian.
\end{cor}
\begin{proof}
By Theorem \ref{thm:einstein_dim}, $\frg$ is nilpotent.
A nilpotent Lorentzian metric Lie algebra is abelian
(Medina \cite{medina}).
\end{proof}

Clearly, for any given dimension $n\geq 6$,
equation \eqref{eq:kappa_square} has many different
solutions which will lead to non-isometric
Lie algebras.

A classification of metric Lie algebras with Witt index $2$
was given by Kath and Olbrich \cite{KO1}. They identified three
classes of solvable but non-abelian indecomposable metric Lie algebras with
index $2$ (Examples 6.1 to 6.3 in \cite{KO1}).
Two of these classes (Examples 6.2 and 6.3 in \cite{KO1}) have
only purely imaginary eigenvalues and thus cannot satisfy
the Einstein condition \eqref{eq:kappa_square}.
The remaining class provides the only candidate for non-abelian
Einstein Lie algebras with index $2$.
It generalizes Example \ref{ex:kappa=0}.

\begin{example}\label{ex:KO1}
Consider the vector space $\frg=\fra\oplus\ab^{n-1}_{s-1}\oplus\frj$
with $\fra=\Span\{a\}$, $\frj=\Span\{z\}$, and extend the
pseudo-Riemannian scalar product on $\ab^{n-1}_{s-1}$ to $\frg$ by
\[
\langle a,a\rangle = \langle z,z\rangle =0,
\quad
\langle a,z\rangle = 1,
\quad
(\fra\oplus\frj)\perp\ab^{n-1}_{s-1}.
\]
Define a Lie product on $\frg$ by
\[
[a,x] = \delta_a(x),
\quad
[x,y]=\langle[a,x],y\rangle z
\]
for some $\delta_a\in\sso_{n-s-1,s-1}$, and $x,y\in\ab^{n-1}_{s-1}$,
all other relations zero.
Then $\met$ is an invariant scalar product of signature
$(n-s,s)$ on $\frg$.
The scalar product $\met$ is Einstein if and only if $\delta_a$
satisfies \eqref{eq:kappa_square}.
The Lie algebra $\frg$ is solvable, and it is nilpotent if
and only if $\delta_a$ is nilpotent.
For $s=2$, this is the only class of Lie algebras with invariant
Einstein scalar products.
\end{example}

%%%%
%\subsection{Einstein double extension}
%
%Suppose $(\frg,\met)$ is the double extension of the solvable
%metric Lie algebra $(\frg_1,\met_1)$.
%As vector spaces, we write $\frg=\fra\oplus\frw\oplus\frj$
%as in \eqref{eq:decompositionW}, where we identify $\frw$ with
%$\frg_1$.
%Let $\delta_a$ denote the derivation by which $a\in\fra$ acts on
%$\frg_1$.
%
%One can easily check the following
%(as already observed by Baum and Kath \cite{BK}):
%
%\begin{lem}\label{lem:einstein_double_extension}
%$\met$ is an invariant Einstein scalar product on $\frg$
%if and only if the following three conditions holds:
%\begin{enumerate}
%\item
%$\kappa_{\frg_1}=0$.
%\item
%$\tr(\ad_{\frg_1}(x)\circ\delta_a)=0$ for all $x\in\frg_1$,
%$a\in\fra$.
%\item
%$\tr(\delta_{a_1}\circ\delta_{a_2})=0$ for all $a_1,a_2\in\fra$.
%\end{enumerate}
%\end{lem}
%

%%%
\subsection{An algebraic lemma}

We prove an algebraic result to help us with the existence
question of lattices in low dimensions.

We first recall a famous result from the theory of transcendental
numbers that solved Hilbert's Seventh Problem
(see Waldschmidt \cite{waldschmidt} for reference).

\begin{gelfondschneider}
Let $\alpha\in\CC\backslash\{0,1\}$ and let $\beta\in\CC$
be irrational.
Then at least one of the numbers $\alpha$, $\beta$ and $\alpha^\beta$
is transcendental.
\end{gelfondschneider}

\begin{lem}\label{lem:algebraic}
Let $X$ be a matrix of the form \eqref{eq:PWZ}
with eigenvalues
$\lambda_1,\ldots,\lambda_k\in\RR$ and
$\zeta_1,\ldots,\zeta_m,\bar{\zeta}_1,\ldots,\bar{\zeta}_m\in\CC\backslash\RR$.
Suppose $X$ is not nilpotent and the eigenvalues satisfy
\begin{equation}
\lambda_1^2 + \ldots + \lambda_k^2
+2\Re(\zeta_1)^2+\ldots+2\Re(\zeta_m)^2
-2\Im(\zeta_1)^2-\ldots-2\Im(\zeta_m)^2
=0.
\label{eq:eigenvalues}
\end{equation}
If $n\leq 5$, then there is no $t\in\RR$ such that
$\exp(tX)$ is conjugate to a matrix in $\SL(n,\QQ)$.
\end{lem}
\begin{proof}
For the eigenvalues to satisfy the given equation, a non-nilpotent
$X$ must have at least one non-real eigenvalue pair $\zeta$,
$\bar{\zeta}$.
If $n\leq 5$, this is only possible for $n=4$ and $n=5$.
We distinguish two cases:
\begin{enumerate}
\item
If $\Re(\zeta)\neq0$, then the form of $X$ implies
that $-\zeta$ is another eigenvalue.
So the non-zero eigenvalues of $X$ are
$\zeta,\bar{\zeta},-\zeta,-\bar{\zeta}$.
The condition \eqref{eq:eigenvalues} becomes
$\Re(\zeta)^2=\Im(\zeta)^2$, so if $\alpha=\Re(\zeta)$,
then the eigenvalues of $\exp(X)$ other than $1$ are
\[
\e^{\alpha(1+\i)},\e^{\alpha(1-\i)},\e^{\alpha(-1+\i)},\e^{\alpha(-1-\i)}.
\]
\item
Suppose $\Re(\zeta)=0$. Then we either have an additional real
eigenvalue $\lambda$ (hence also $-\lambda$) or two eigenvalues in
$\zeta,\zeta'\in\i\RR$ (otherwise we would be in case (1)).
In the latter case, condition \eqref{eq:eigenvalues}
cannot be satisfied. So we may assume that there exists an
eigenvalue $\lambda\in\RR$. By \eqref{eq:eigenvalues},
$\lambda^2=\zeta^2$.
So the eigenvalues of $\exp(X)$ other than $1$ are
\[
\e^{\lambda}, \e^{-\lambda}, \e^{\i\lambda}, \e^{-\i\lambda}.
\]
\end{enumerate}
Note that in both cases (1) and (2) for every eigenvalue $\xi$ of
$\exp(X)$, the number $\xi^\i$ is also an eigenvalue.

Assume there is some $t\in\RR$ such that $\exp(tX)$ is
conjugate to an matrix in $\SL(n,\QQ)$.
As zeros of the characteristic polynomial of a matrix in
$\SL(n,\QQ)$, every eigenvalue $\xi^t$ of $\exp(tX))$ is
algebraic. However, if $\xi^t$ is algebraic, then,
since $\i$ is algebraic and irrational,
the Gelfond-Schneider Theorem
implies that $(\xi^t)^\i=(\xi^{\i})^t$ is transcendental.
But as noted above, this number is also an eigenvalue of
$\exp(tX)$ and hence algebraic, a contradiction.
\end{proof}

It seems the argument using the Gelfond-Schneider Theorem
cannot prove the statement for dimension higher than $5$.

\begin{example}
Consider the following matrix $X\in\sso(5,1)$ of the form \eqref{eq:PWZ},
\[
X
=
\begin{pmatrix}
\lambda & 0& & & & *  \\
0&-\lambda & & &  \\
&& 0& -\beta_1&   & \\
&& \beta_1&0 & &   \\
& && & 0&-\beta_2  \\
0& && & \beta_2&0  
\end{pmatrix}.
\]
Condition \eqref{eq:eigenvalues} yields
$\lambda=\pm\sqrt{\beta_1^2+\beta_2^2}$, and so the eigenvalues
of $\exp(X)$ are
\[
\e^{\i\beta_1},\ \e^{-\i\beta_1},\ \e^{\i\beta_2},\
\e^{-\i\beta_2},\ \e^{\sqrt{\beta_1^2+\beta_2^2}},\
\e^{-\sqrt{\beta_1^2+\beta_2^2}}.
\]
It is not clear how these relations between the eigenvalues
could give a contradiction to their algebraicity.
\end{example}

However, a generalization of Lemma \ref{lem:algebraic} to any
$n\in\NN$ would hold if a famous conjecture on
transcendental numbers turned out to be true.

\begin{schanuel}
Let $\alpha_1,\ldots,\alpha_d$ be complex numbers that are linearly
independent over $\QQ$.
Then the transcendence degree over $\QQ$ of the extension field
$\QQ(\alpha_1,\ldots,\alpha_d,\e^{\alpha_1},\ldots,\e^{\alpha_d})$
is at least $d$.
\end{schanuel}

In fact, let $\{\alpha_1,\ldots,\alpha_d\}$ is a subset of
eigenvalues of $X$ that is maximally
linearly independent over $\QQ$.
We can write $\Re(\zeta_i)=\frac{1}{2}(\zeta_i+\bar{\zeta}_i)$
and $\Im(\zeta_i)=\frac{\i}{2}(\bar{\zeta}_i-\zeta_i)$
in \eqref{eq:eigenvalues}.
Since for every eigenvalue $\zeta\in\CC\backslash\RR$,
the conjugate $\bar{\zeta}$ is also an eigenvalue, this shows that
the eigenvalues of $X$ satisfy an algebraic relation over $\QQ$.
Replacing then every eigenvalue in \eqref{eq:eigenvalues} by a
linear combination of the $\alpha_1,\ldots,\alpha_d$ shows
that $\alpha_1,\ldots,\alpha_d$ satisfy an algebraic relation
over $\QQ$, and therefore do not form an algebraically
independent set.
If Schanuel's Conjecture is true, then at least one of the
exponentials $\e^{\alpha_1},\ldots,\e^{\alpha_d}$ is
transcendental. Therefore:

\begin{observation}\label{obs:schanuel}
If Schanuel's Conjecture is true, then Lemma \ref{lem:algebraic}
holds without the assumption that $n\leq 5$.
\end{observation}

%%%
\subsection{Compact Einstein solvmanifolds}

Now we investigate which solvable Lie groups $G$ can give rise
to compact Einstein solvmanifolds $M=G/\Gamma$.
A necessary condition is the existence of an invariant
Einstein scalar product on the Lie algebra $\frg$ of $G$.
We first observe a direct consequence of Corollary
\ref{cor:Lorentz_Einstein}:

\begin{cor}
A compact homogeneous Lorentzian Einstein solvmanifold is flat.
\end{cor}

In Example \ref{ex:kappa=0} we saw that non-nilpotent solvable
Lie groups $G$ with Einstein metrics exist.
The question is whether these Lie groups can give rise to compact
homogeneous quotients with Einstein metrics.
This requires the existence of a lattice $\Gamma\leq G$.
%For a simply connected solvable Lie group $G$ to admit a lattice,
%necessary conditions due 
%to Auslander \cite[III.6]{auslander} are the existence of a
%$\QQ$-structure on $\frn$ and that $\exp(\ad(\fra)|_{\frn})$ contains a lattice
%$\Lambda$ whose adjoint action of on $\frn$ can be represented by
%integer matrices.
%The second of these conditions implies that for each
%$\ad(a)\in\fra\cap\exp^{-1}(\Lambda)$,
%the characteristic polynomial of $\exp(\ad(a))=\Ad(\exp(a))$ has
%coefficients in $\ZZ$.
%This means that all the eigenvalues
%\[
%\e^{\lambda_i}, \e^{\zeta_i}, \e^{\bar{\zeta}_i}
%\]
%of $\exp(\ad(a))$ are algebraic numbers.
We may assume that $G$ is connected and simply connected.
A theorem by Mostow \cite[Theorem 5.5]{mostow70} then implies that
we can assume $G$ is linear and $\Gamma$ is given by matrices
with integer coefficients. In particular, the adjoint
action of $\Gamma$ on the nilradical $\frn$ of $\frg$ is given
by rational matrices.
If $G$ is not nilpotent, this means that for
a vector space complement $\fra$ of $\frn$ in $\frg$, there exist
non-zero $a\in\fra$ such that the
characteristic polynomial of $\exp(\ad(a))=\Ad(\exp(a))$ has
coefficients in $\QQ$.
This means that all the eigenvalues
\[
\e^{\lambda_i}, \e^{\zeta_i}, \e^{\bar{\zeta}_i}
\]
of $\exp(\ad(a))$ are algebraic numbers.
In dimensions $\leq 7$ this is impossible:
{
\renewcommand{\thethm}{\ref{thm:no_einstein_dim6}}
\begin{thm}
Let $M$ be a compact pseudo-Riemannian Einstein solvmanifold of
dimension less or equal to $7$.
Then $M$ is a nilmanifold.
\end{thm}
\addtocounter{thm}{-1}
}

\begin{proof}%[Proof for dimension $6$]
$M$ is of the form $M=G/\Gamma$, where $\Gamma$ is a lattice in $G$.
Assume that $G$ is not nilpotent.
By Theorem \ref{thm:einstein_dim} we only need to consider the case
where $G$ is a six- or seven-dimensional solvable Lie group.
We may further assume that $G$ is simply
connected.

The Einstein metric on $M$ induces an Einstein scalar product $\met$
on the Lie algebra $\frg$ of $G$.
Consider the vector space decomposition $\frg=\fra\oplus\frn$ from \eqref{eq:g=a+n}.
As a consequence of Theorem \ref{thm:einstein_dim}, $6\geq\dim\frn\geq5$
and $2\geq\dim\fra\geq1$.
All elements $x\in\frn$ satisfy the Einstein condition
\eqref{eq:kappa_square}.

Let $0\neq a\in\fra$, and let $\fri$ be a non-zero
central ideal of $\frg$ contained in $\zen(\frn)\cap[\frg,\frn]$
(this exists by Proposition \ref{prop:j0}
because $\frg$ is not abelian).
Then $\dim\ker\ad(a)\geq 2$,
because $\ker\ad(a)\supseteq\fra\oplus\fri$.
Also, $\ad(a)$ is not nilpotent and $W_1=\im\ad(a)$ is at most
five-dimensional.
%and the restriction $\ad(a)|_{W_1}$ must have
%at least one non-real eigenvalue pair $\zeta$, $\bar{\zeta}$
%in order to satisfy \eqref{eq:kappa_square}.
%We distinguish two cases:
%\begin{enumerate}
%\item
%If $\Re(\zeta)\neq0$, then skew-symmetry implies
%that $-\zeta$ is another eigenvalue.
%So the eigenvalues of $\ad(a)$ are
%$\zeta,\bar{\zeta},-\zeta,-\bar{\zeta}$.
%The Einstein condition \eqref{eq:kappa_square} becomes
%$\Re(\zeta)^2=\Im(\zeta)^2$, so if $\alpha=\Re(\zeta)$,
%then the eigenvalues of $\exp(\ad(a))$ other than $1$ are
%\[
%\e^{\alpha(1+\i)},\e^{\alpha(1-\i)},\e^{\alpha(-1+\i)},\e^{\alpha(-1-\i)}.
%\]
%\item
%Suppose $\Re(\zeta)=0$. Then we either have an additional real
%eigenvalue $\lambda$ (hence also $-\lambda$) or two eigenvalues in
%$\zeta,\zeta'\in\i\RR$ (otherwise we would be in case (1)).
%In the latter case, the Einstein condition \eqref{eq:kappa_square}
%cannot be satisfied. So we may assume that there exists an
%eigenvalue $\lambda\in\RR$. By \eqref{eq:kappa_square},
%$\lambda^2=\zeta^2$.
%So the eigenvalues of $\exp(\ad(a))$ other than $1$ are
%\[
%\e^{\lambda}, \e^{-\lambda}, \e^{\i\lambda}, \e^{-\i\lambda}.
%\]
%\end{enumerate}
%Note that in both cases (1) and (2) for every eigenvalue $\xi$ of
%$\exp(\ad(a))$, the number $\xi^\i$ is also an eigenvalue.
%

By Mostow's theorem, there exists a non-zero $a\in\fra$ such
that $\exp(\ad(a))$ is conjugate to a
matrix in $\SL(n,\ZZ)$, where $n=6$ or $7$.
But application of Lemma \ref{lem:algebraic} to $\ad(a)|_{W_1}$
shows that $\exp(t\ad(a))$ is not conjugate to a matrix in
$\SL(n,\ZZ)$ for any $t\in\RR$.
%As zeros of the characteristic polynomial of a matrix in
%$\SL(n,\ZZ)$, every eigenvalue $\xi^t$ of $\exp(t\ad(a))$ is
%algebraic. However, if $\xi^t$ is algebraic, then,
%since $\i$ is algebraic and irrational,
%the Gelfond-Schneider Theorem\footnote{This theorem states that
%\textit{if $\alpha\in\CC\backslash\{0,1\}$ and $\beta\in\CC$
%are irrational, then at least one of the numbers $\alpha$, $\beta$ and $\alpha^\beta$is transcendental}.} %(Appendix \ref{sec:transcendental})
%implies that $(\xi^t)^\i=(\xi^{\i})^t$ is transcendental.
%But as noted above, this number is also an eigenvalue of
%$\exp(t\ad(a))$ and hence algebraic, a contradiction.
It follows that $G$ must be nilpotent, which proves
the theorem.
\end{proof}

%\begin{proof}[Proof for dimension $7$]
%Let $M=G/\Gamma$ as before, and let $\frg$ be the Lie algebra
%of $G$.
%Let $\frj$ be a one-dimensional central ideal in $\frg$.
%Reduction of $\frg$ by $\frj$ yields a Lie algebra
%$\frg_1=\frj^\perp/\frj$ of dimension $5$, which inherits an
%Einstein scalar product $\met_1$ if $\met$ is Einstein (Lemma \ref{lem:einstein_double_extension}).
%Assume this is the case, so that $\frg_1$ is nilpotent by
%Theorem \ref{thm:einstein_dim}.
%Let $\fra=\Span\{a\}\subset\frg$ such that $\frg=\fra\oplus\frg_1\oplus\frj$, and let $\delta_a=\ad(a)|_{\frg_1}$.
%If the index of $\met_1$ is $0$, then the index of $\met$ is $1$
%and thus $\frg$ is abelian (Lemma \ref{cor:Lorentz_Einstein}).
%
%If the index of $\met_1$ is $1$, then $\delta_a$ has a
%representation of the form \eqref{eq:PWZ}, with one real
%eigenvalue pair $\lambda,-\lambda$. Since in this case the block
%$X_1$ in \eqref{eq:PWZ} lies in $\sso_3$, we have one eigenvalue
%$0$ and one pair of purely imaginary eigenvalues $\i\beta,-\i\beta$. By \eqref{eq:kappa_square}, $\lambda^2=\beta^2$.
%This yields a contradiction to the existence of a lattice in $G$
%as in the proof for dimension $6$,
%unless $\delta_a$ is nilpotent.
%
%If the index of $\met_1$ is $2$, then $\delta_a$ has a
%representation of the form \eqref{eq:PWZ} with complex
%eigenvalues $\zeta, \bar{\zeta}, -\zeta, -\bar{\zeta}$.
%The additional eigenvalue is $0$. As in the proof for dimension
%$6$, we get a contradiction to the existence of a lattice in $G$,
%unless $\delta_a$ is nilpotent.
%In any of the above cases, $\frg$ is nilpotent.
%\end{proof}

Observation \ref{obs:schanuel} suggests the following:

{
\renewcommand{\thethm}{\ref{conj:no_einstein}}
\begin{conjecture}
Every compact pseudo-Riemannian Einstein solvmanifold
is a nilmanifold.
\end{conjecture}
\addtocounter{thm}{-1}
}

%\begin{conjecture}
%Every pseudo-Riemannian homogeneous Einstein manifold of finite
%volume is either a quotient of a semisimple Lie group of isometries
%or a nilmanifold.
%\end{conjecture}

\begin{remark}
Note that finding a counterexample $X$ to the statement of
Lemma \ref{lem:algebraic} for dimensions $n\geq 6$ would allow
us to construct a Lie algebra $\frg$ with Einstein scalar product
as in Example \ref{ex:KO1} with $\delta_a=X$.
The simply connected Lie group $G$ with Lie
algebra $\frg$ would not be nilpotent and would contain a lattice
$\Gamma=\exp(tX)\ltimes\ZZ^{n+1}$, for a certain $t\in\RR$.
This tells us that Conjecture \ref{conj:no_einstein} is in
fact equivalent to the statement of Lemma \ref{lem:algebraic} for all $n$.
\end{remark}

% !TEX root = einsteinCompactSolvable.tex

%%%%%%%%%%%%%%%%%%%%%%%
\section{Semisimple isometry groups}
%%%%%%%%%%%%%%%%%%%%%%%
\label{sec:semisimple}

We now study homogeneous pseudo-Riemannian Einstein
manifold of finite volume $M=G/H$ where $G$ is semisimple.
The metric $\g_M$ on $M$ induces a symmetric bilinear form $\met$
on $\frg$ with metric
radical $\frg^\perp=\frh$, the Lie algebra of $H$.
As usual, $G$ acts almost effectively, so that
$\frh$ contains no non-zero ideals of $\frg$.

Assume that $G$ is simply connected. Then $G$ is a product
\begin{equation}
G=K\times S,
\label{eq:GKS}
\end{equation}
where $K$ is compact semisimple and $S$ is semisimple without
compact factors. Let $\frk$, $\frs$ denote their respective
Lie algebras.

%\begin{thm}\label{thm:semisimple_stabilizer}
%Let $H^\circ$ denote the identity component of $H$.
%Then $H^\circ\subset K$.
%Let $\pi_S$ denote the projection map on $S$.
%Then $\pi_S(H)$ is a lattice in $S$.
%\end{thm}

\begin{lem}\label{lem:S_properties}
$\met$ is $\frs$-invariant, $\frk\perp\frs$,
and $\frs\cap\frg^\perp=\zsp$.
In particular, if the metric $\g$ on $M$ is Einstein, then the
restriction $\met_\frs$ of $\met$ to $\frs$ is Einstein and thus
a multiple of the Killing form of $\frs$.
\end{lem}
\begin{proof}
The invariance of $\met$ is part of Theorem \ref{thm:invariance}
(it also follows immediately from Mostow \cite[Lemma 3.1 (1)]{mostow3}).

To see the orthogonality of $\frk$ and $\frs$,
let $k\in\frk$ and $s\in\frs$. We may assume $s=[s_1,s_2]$ for
some $s_1,s_2\in\frs$.
Then
$\langle k,s\rangle = \langle k,[s_1,s_2]\rangle
=-\langle[s_1,k],s_2\rangle =0$. So $\frk\perp\frs$.

Let $x_0\in\frg^\perp$ and write $x_0=k_0+s_0$ with $k_0\in\frk$,
$s_0\in\frs$.
Then $s_0\perp\frs$ since $k_0\perp\frs$, as was just shown.
Since also $s_0\perp\frk$, it follows that $s_0\in\frg^\perp$.
This means $\pi_{\frs}(\frg^\perp)=\frg^\perp\cap\frs$, where
$\pi_{\frs}:\frg\to\frs$ denotes the canonical projection.
Then $\frg^\perp\cap\frs$ is an ideal in $\frg$, as $\met$ is
$\frs$-invariant and $[\frk,\frs]=\zsp$.
But $\frg^\perp$ contains no non-zero ideals of $\frg$, so
$\pi_{\frs}(\frg^\perp)=\frg^\perp\cap\frs=\zsp$, which means
$\frg^\perp\subset\frk$.
\end{proof}

As $K$ is compact, the projection $\pi_S(H)$ of $H$ to $S$ is a
closed subgroup of $S$. Then the above lemma and the Borel Density
Theorem imply:

\begin{prop}\label{prop:semisimple_stabilizer}
Let $H^\circ$ denote the identity component of $H$.
Then $H^\circ\subset K$.
Let $\pi_S$ denote the projection map to $S$.
Then $\pi_S(H)$ is a lattice in $S$.
\end{prop}

%For proofs of the above, see Baues, Globke and Zeghib \cite{BGZ}.

{
\renewcommand{\thethm}{\ref{thm:Einstein_semisimple}}
\begin{thm}
Let $M$ be a pseudo-Riemannian homogeneous Einstein manifold
of finite volume
and $G$ a semisimple, connected and simply connected Lie group
acting transitively and almost effectively by isometries on $M$.
Identify $M$ with $G/H$ for a closed subgroup $H$ of $G$.
Let $G=K\times S$ be the decomposition of $G$ into a compact
semisimple factor $K$ and a semisimple factor $S$ without
compact simple factors.
Then:
\begin{enumerate}
\item
The identity component $H^\circ$ of $H$ is contained in $K$.
\item
$M=\hat{M}/\Gamma$, where $\hat{M}$ is a pseudo-Riemannian product of
Einstein manifolds
\[
\hat{M}=M_K\times S,
\]
with $M_K=K/(H\cap K)$,
and the Einstein metric on $S$ is a multiple of the Killing form on $S$.
\item
$\Gamma$ is the graph of a homomorphism
from a lattice in $S$ into $K$.
\end{enumerate}
\end{thm}
\addtocounter{thm}{-1}
}
\begin{proof}
By Proposition \ref{prop:semisimple_stabilizer}, $H^\circ\subset K$
and $H$ projects to a lattice $\pi_S(H)\cong H/(H\cap K)$ in $S$.
Hence $\hat{M}=M_K\times S$ is a cover of $M$, that is,
$M=\hat{M}/\Gamma$ for some discrete subgroup $\Gamma$ of $G$.
This $\Gamma$ is isomorphic to $\pi_S(H)$ and thus necessarily of the form
required in the theorem.
Moreover, the metric $\g_M$ pulls back to
a $G$-invariant pseudo-Riemannian Einstein metric $\hat{\g}$ on
$\hat{M}=M_K\times S$.

By $\frk\perp\frs$ (Lemma \ref{lem:S_properties}), by the
$G$-invariance of $\hat{\g}$ and because $S$ and $K$ commute,
the $K$-orbits and the $S$-orbits are orthogonal at every point
in $\hat{M}$.
This means $\hat{M}=M_K\times S$ is a pseudo-Riemannian product.
As $\hat{M}$ is an Einstein manifold, both of
its orthogonal factors $M_K$ and $S$ are Einstein manifolds.
Since the pullback $\g_S$ of the metric $\g_M$ to $S$ is non-degenerate and
bi-invariant (Theorem \ref{thm:invariance}, Corollary \ref{cor:no_compact}),
$S$ itself is an Einstein manifold and hence $\g_S$ is given by a
multiple of the Killing form.
\end{proof}

%\appendix

%\input{inputComputations}

%\input{inputDim4Invariant}

%%%%% Bibliography %%%%%%

% !TEX root = einsteinCompactSolvable.tex

%%%%% Bibliography %%%%%%

\end{document}